\documentclass[11pt,reqno]{amsart}

\usepackage[T1]{fontenc}
\usepackage[utf8]{inputenc}
\usepackage{lmodern}
\usepackage{microtype}
\usepackage{mathtools}
\usepackage{amssymb}
\usepackage{xcolor}
\usepackage[a4paper,margin=1.12in]{geometry}
\usepackage[
  colorlinks=true,
  linkcolor=blue!50!black,
  citecolor=blue!50!black,
  urlcolor=blue!50!black
]{hyperref}

\newtheorem{theorem}{Theorem}[section]
\newtheorem{proposition}[theorem]{Proposition}
\newtheorem{lemma}[theorem]{Lemma}
\newtheorem{corollary}[theorem]{Corollary}
\theoremstyle{remark}

\newtheorem{question}[theorem]{Question}

\title{Bang vectors in the real polarization problem:\\ the exact dimensional range}

\author{Dami\'an Pinasco}
\address{Universidad Torcuato Di Tella, Departamento de Matem\'atica y Estad\'istica, and CONICET, Buenos Aires, Argentina}
\email{dpinasco@utdt.edu}


\begin{document}

\begin{abstract}
We determine the exact dimensional range in which a normalized longest signed sum, which we call a Bang vector, always satisfies the real polarization inequality: this prescription is universally valid if and only if \(n\leq 14\). For every \(n\geq 15\) we construct \(n\) unit vectors whose all-positive sum is, up to global sign, the unique longest signed sum, while its normalized direction has polarization product smaller than \(n^{-n/2}\). The counterexamples are given by a single explicit family, and for the same configurations we exhibit nonmaximal signed sums whose normalized directions do satisfy the polarization bound.
\end{abstract}

\maketitle

\section{Introduction}

The real polarization conjecture, formulated by Ben\'itez, Sarantopoulos and Tonge in the late 1990s~\cite{BenitezSarantopoulosTonge1998}, remained open for almost three decades. During that period, several exact low-dimensional results were obtained by locating explicit points on the sphere. The present note revisits one of those concrete methods and determines its exact dimensional range.

Let $v_1,\dots,v_n$ be unit vectors in $\mathbb R^n$ and set
\[
P(x)=\prod_{i=1}^n |\langle x,v_i\rangle|,
\qquad x\in S^{n-1}.
\]
The real polarization conjecture asserts that
\begin{equation}\label{eq:polarization}
\max_{x\in S^{n-1}}P(x)\geq n^{-n/2}.
\end{equation}
Equivalently, the $n$-th linear polarization constant of $\mathbb R^n$ is $n^{n/2}$. Mart\'inez and Ortega-Moreno recently proved the stronger conjecture of Ball and Frenkel and thereby established \eqref{eq:polarization} in every dimension~\cite{MartinezOrtega2026}.

The general existence problem is therefore settled. A separate question concerns the particular constructions that were used to produce suitable points. One of the exact methods in low dimensions was to test the product at a normalized signed sum of maximal Euclidean norm. The connection between plank problems and polarization was developed by R\'ev\'esz and Sarantopoulos~\cite{ReveszSarantopoulos2004}. Motivated more specifically by Bang's proof of Tarski's plank theorem~\cite{Bang1951}, choose signs $\varepsilon_i\in\{-1,1\}$ so that
\[
s_{\varepsilon}=\sum_{i=1}^n\varepsilon_i v_i
\]
has maximal Euclidean norm, and test \eqref{eq:polarization} at $s_{\varepsilon}/\|s_{\varepsilon}\|$. We refer to any such normalized sum as a \emph{Bang vector} of the system.

Pappas and R\'ev\'esz proved that the normalized longest signed sum satisfies the polarization inequality for $n\leq5$. They then asked whether, in every dimension, some choice of signs always produces a normalized signed sum satisfying the inequality~\cite{PappasRevesz2004}. Matolcsi and Mu\~noz later showed that the particular longest-sum prescription is not universal: they constructed a counterexample in dimension $34$, and the same mechanism yields examples in every higher dimension~\cite{MatolcsiMunoz2006}. Their result does not settle the signed-sum selection question; in their example, a different nonmaximal sign choice still satisfies the inequality. The positive range for the longest-sum prescription was subsequently extended to $n\leq14$~\cite{Pinasco2023}. Thus, before the present work, it remained unknown whether this prescription was universally valid in dimensions $15\leq n\leq33$.

The proof of the $n\leq14$ result is governed by a scalar inequality whose direction reverses at $n=15$. A closer examination of that inequality suggests the construction developed below. The resulting family shows that the reversal reflects a genuine geometric obstruction rather than an artifact of the earlier argument. Unlike the earlier counterexample, which combines a concentrated block with additional orthogonal directions, the present examples are built directly in each dimension. For every $n\geq15$, the components orthogonal to a distinguished axis form a regular simplex in an $(n-2)$-dimensional subspace. Their sum vanishes, so the all-positive signed sum points along the axis, while the symmetry of the simplex reduces the comparison with every other signed sum to a scalar condition.

A unit vector $e$ plays the role of the axis, and $n-1$ further vectors lie at the same angle from $e$. The resulting picture is that of a simplicial umbrella: $e$ is its handle and the remaining vectors are symmetrically arranged ribs. The parameter controlling their angle with $e$ balances two opposite demands. Moving the ribs away from the handle lowers the product in the handle direction, but moving them too far destroys the maximality of the corresponding signed sum. These two requirements are incompatible in dimension $14$ and become compatible from dimension $15$ onward.

\begin{theorem}\label{thm:main}
The Bang-vector prescription is universally valid for systems of $n$ unit vectors in $\mathbb R^n$ if and only if $n\leq14$.

More precisely, for every integer $n\geq15$ there exist unit vectors $v_1,\dots,v_n\in\mathbb R^n$ such that, up to a simultaneous change of all signs,
\[
s=\sum_{i=1}^n v_i
\]
is the unique longest signed sum and
\[
\prod_{i=1}^n
\left|
\left\langle \frac{s}{\|s\|},v_i\right\rangle
\right|
<n^{-n/2}.
\]
\end{theorem}

Thus the longest-sum prescription has a sharp dimensional threshold at \(n=15\). The final part of the paper returns to the signed-sum selection question within the same family. Although the longest signed sum fails, suitable nonmaximal signed sums still satisfy the polarization bound. We first show that almost balanced patterns with rib imbalance \(q_n=x\sqrt n+O(1)\) have a simple limiting profile, with a transition at \(x=\sqrt e-1\). We then give, for every \(n\geq15\), an explicit successful sign pattern by taking the smallest admissible imbalance not below \(\sqrt n\).

\section{The simplex-umbrella construction}

Fix $n\geq3$, put $m=n-1$, and let $e\in\mathbb R^n$ be a unit vector. For $0<a<1$, choose vectors $w_1,\dots,w_m\in e^\perp$ forming a centered regular simplex, with its scale chosen so that $\|w_i\|^2=1-a^2$. The resulting geometric data are
\begin{equation}\label{eq:simplex}
\sum_{i=1}^m w_i=0,
\qquad
\|w_i\|^2=1-a^2,
\qquad
\langle w_i,w_j\rangle=-\frac{1-a^2}{m-1}
\quad(i\neq j).
\end{equation}
Here the first identity expresses that the simplex is centered, while the last follows from regularity together with the chosen scale.
Such a simplex lies in an $(m-1)$-dimensional subspace of $e^\perp$. Define
\begin{equation}\label{eq:vectors}
v_i=ae+w_i\quad(1\leq i\leq m),
\qquad
v_n=e.
\end{equation}
All the $v_i$ are unit vectors. Since the components orthogonal to $e$ cancel,
\begin{equation}\label{eq:s}
s:=\sum_{i=1}^n v_i=(ma+1)e.
\end{equation}
Consequently, $s/\|s\|=e$ and
\begin{equation}\label{eq:product}
P(e)=a^m=a^{n-1}.
\end{equation}
Thus a smaller value of $a$ makes the product smaller. The remaining issue is to determine how small $a$ may be while $s$ remains the longest signed sum.

We shall require strict maximality. Besides simplifying the comparison among sign patterns, this ensures that $s$ is, up to a global change of sign, the unique longest signed sum. Thus the failure exhibited below does not depend on how one chooses among several maximizers.

For two distinct ribs $v_i$ and $v_j$, the orthogonality of $e$ and the simplex subspace gives
\begin{align*}
\langle v_i,v_j\rangle
&=\langle ae+w_i,ae+w_j\rangle\\
&=a^2\langle e,e\rangle
  +a\langle e,w_j\rangle
  +a\langle w_i,e\rangle
  +\langle w_i,w_j\rangle\\
&=a^2+\langle w_i,w_j\rangle\\
&=a^2-\frac{1-a^2}{m-1}\\
&=\frac{ma^2-1}{m-1}.
\end{align*}
We therefore set
\begin{equation}\label{eq:lambda}
\lambda:=\frac{1-ma^2}{m-1},
\qquad
\langle v_i,v_j\rangle=-\lambda
\quad(i\neq j\leq m).
\end{equation}
In the range $ma^2<1$ considered below, $\lambda>0$, so distinct ribs have common negative scalar product $-\lambda$. Small values of $\lambda$ mean that the ribs are nearly pairwise orthogonal.

For every nonempty proper subset $A\subset\{1,\dots,n\}$, set
\[
s_A=s-2\sum_{i\in A}v_i.
\]
Then
\begin{equation}\label{eq:cut}
\|s_A\|^2
=
\|s\|^2
-4\sum_{\substack{i\in A\\ j\notin A}}
\langle v_i,v_j\rangle.
\end{equation}
Thus $s$ is strictly longer than $s_A$ whenever the scalar-product sum across the cut $A\mid A^c$ is positive.

Let
\[
k=\#\bigl(A\cap\{1,\dots,m\}\bigr)
\]
be the number of ribs in $A$. There are two cases. If $n\notin A$, then $A$ consists of $k$ ribs. For each such rib, the handle $v_n=e$ contributes $a$ across the cut, while the $m-k$ ribs in the complement contribute $-\lambda$ each. Hence
\[
\sum_{\substack{i\in A\\ j\notin A}}
\langle v_i,v_j\rangle
=
k\bigl(a-(m-k)\lambda\bigr).
\]
If $n\in A$, it is more convenient to count from the complement, which contains $m-k$ ribs. Each such rib has scalar product $a$ with the handle and scalar product $-\lambda$ with each of the $k$ ribs of $A$. Therefore
\[
\sum_{\substack{i\in A\\ j\notin A}}
\langle v_i,v_j\rangle
=
(m-k)\bigl(a-k\lambda\bigr).
\]
In every nontrivial case, the number of negative contributions $-\lambda$ accompanying a positive contribution $a$ is at most $m-1$. It is therefore enough to impose the uniform worst-case condition
\begin{equation}\label{eq:maximality}
a-(m-1)\lambda
=
a-(1-ma^2)
=
ma^2+a-1>0.
\end{equation}
Under \eqref{eq:maximality}, every nontrivial cut sum is positive, and hence $s$ is the unique longest signed sum up to global sign.

The boundary case
\[
ma^2+a-1=0
\]
corresponds to ties among maximizing sign patterns: flipping any one of the first $m$ signs leaves the norm unchanged. We work instead with the strict inequality \eqref{eq:maximality}, which places the configuration inside the region of unique maximality.

We now combine the two requirements on $a$. On the one hand, $s$ must remain the unique longest signed sum; on the other, its normalized direction $e$ must fail the polarization bound. Thus the umbrella must satisfy the two opposite inequalities
\begin{equation}\label{eq:two-conditions}
ma^2+a-1>0,
\qquad
a^{n-1}<n^{-n/2}.
\end{equation}
Equivalently,
\begin{equation}\label{eq:window-a}
\alpha_n<a<\beta_n,
\qquad
\alpha_n:=\frac{\sqrt{4n-3}-1}{2(n-1)},
\qquad
\beta_n:=n^{-\frac{n}{2(n-1)}}.
\end{equation}
The left endpoint is the threshold for strict maximality of the all-positive sum, while the right endpoint is the threshold below which its direction fails the polarization bound.

The two endpoints in \eqref{eq:window-a} are already implicit in the proof of the positive result for $n\leq14$. The relevant scalar inequality holds up to dimension $14$ and reverses at $n=15$. The construction below shows that this reversal is not merely a limitation of the argument: once the interval $(\alpha_n,\beta_n)$ becomes nonempty, it can be realized geometrically by a family for which the Bang vector fails.

\subsection{The parameter scale}

The order of $a$ is forced by the two inequalities in \eqref{eq:two-conditions}. To see the leading scale, write
\[
a=\frac{\kappa_n}{\sqrt n}.
\]
Then
\[
(n-1)a^2+a-1=\kappa_n^2-1+o(1).
\]
Thus maximality requires $\liminf \kappa_n\geq1$. On the other hand,
\[
n^{n/2}a^{n-1}=\sqrt n\,\kappa_n^{\,n-1}.
\]
If $\kappa_n$ stays larger than $1$ by a fixed amount, this expression grows exponentially, so the direction $e$ cannot violate the polarization bound. The two requirements therefore force
\[
\kappa_n\longrightarrow1,
\qquad\text{or equivalently}\qquad
a\sim n^{-1/2}.
\]

We next look at the first correction to this scale. Motivated by the preceding balance, put
\begin{equation}\label{eq:param-c}
a=\frac1{\sqrt n+c}.
\end{equation}
For fixed $c$,
\[
a
=
\frac1{\sqrt n}
-\frac{c}{n}
+O(n^{-3/2}),
\]
so $c>0$ measures a displacement of order $n^{-1}$ below $n^{-1/2}$. The asymptotic analysis suggests choosing $c$ in a range for which, in high dimensions, both competing conditions become favorable.

Substituting \eqref{eq:param-c} into the two conditions gives
\begin{equation}\label{eq:taylor-max}
(n-1)a^2+a-1
=
\frac{1-2c}{\sqrt n}+O\!\left(\frac{1+c^2}{n}\right),
\end{equation}
and
\begin{equation}\label{eq:taylor-failure}
\log\bigl(n^{n/2}a^{n-1}\bigr)
=
\frac12\log n-c\sqrt n+\frac{c^2}{2}
+O\!\left(\frac{1+c+c^3}{\sqrt n}\right).
\end{equation}
For fixed $c$, the leading term in \eqref{eq:taylor-max} is positive when $c<1/2$. In \eqref{eq:taylor-failure}, every fixed $c>0$ makes the negative term $-c\sqrt n$ dominate the positive term $\frac12\log n$, so the logarithm is eventually negative. Hence the two requirements are asymptotically compatible in the range
\[
0<c<\frac12.
\]

The exact conditions determine a window
\begin{equation}\label{eq:window-c}
\ell_n<c<u_n,
\end{equation}
with lower and upper endpoints
\begin{equation}\label{eq:ell-u}
\ell_n
=
\sqrt n\left(n^{\frac1{2(n-1)}}-1\right),
\qquad
u_n
=
\frac{1+\sqrt{4n-3}}2-\sqrt n.
\end{equation}
These are the images of the exact endpoints $\beta_n$ and $\alpha_n$, respectively, under the change of variables $a=1/(\sqrt n+c)$. Their expansions,
\[
\alpha_n
=
\frac1{\sqrt n}-\frac1{2n}+O(n^{-3/2}),
\qquad
\beta_n
=
\frac1{\sqrt n}
-
\frac{\log n}{2n^{3/2}}
+
O\!\left(\frac{(\log n)^2}{n^{5/2}}\right),
\]
are consistent with \eqref{eq:taylor-max}--\eqref{eq:taylor-failure}.

For $n=14$ one finds
\[
\ell_{14}\approx0.3997>0.3984\approx u_{14},
\]
so the window is empty. For $n=15$ the order is reversed:
\[
\ell_{15}\approx0.3933<0.4019\approx u_{15}.
\]
Thus dimension $15$ is the first one in which the two requirements can be met simultaneously.

\section{A uniform family of counterexamples}

The preceding analysis identifies an admissible parameter window in each dimension. We now seek a single value of $c$ that works throughout the whole range $n\geq15$, thereby producing a uniform family of counterexamples. In the first admissible dimension,
\[
\ell_{15}\approx0.3933
<\frac25
<0.4019\approx u_{15}.
\]
Thus the simple rational choice $c=2/5$ lies strictly inside the window at $n=15$. For each $n\geq15$, set
\begin{equation}\label{eq:choice-a}
a_n=\frac1{\sqrt n+2/5}.
\end{equation}
Since the dimension will always be fixed in what follows, we write simply $a$ for $a_n$ unless the dependence on $n$ needs to be emphasized. The value $2/5$ has no particular geometric significance; it is a convenient point in the first admissible window. We show that this choice remains admissible for every $n\geq15$.

The maximality condition follows directly by substitution. First,
\[
(n-1)a^2<1,
\]
so the corresponding value of $\lambda$ is positive. Moreover,
\begin{equation}\label{eq:delta}
(n-1)a^2+a-1
=
\frac{5\sqrt n-19}
{25(\sqrt n+2/5)^2}>0
\qquad(n\geq15).
\end{equation}
Hence the all-positive sum is uniquely longest, up to global sign, for every umbrella in the family.

We turn to the second requirement, the failure of the polarization bound in the handle direction.

\begin{lemma}\label{lem:scalar}
For every integer $n\geq15$,
\[
\left(1+\frac{2}{5\sqrt n}\right)^{n-1}>\sqrt n.
\]
\end{lemma}

\begin{proof}
Embed the discrete inequality in a one-variable monotonicity problem. For $x\geq\sqrt{15}$ set
\[
h(x)
=(x^2-1)\log\left(1+\frac{2}{5x}\right)-\log x.
\]
Then the desired assertion is $h(\sqrt n)>0$. Using
$\log(1+t)>t-t^2/2$ for $t>0$, we obtain
\begin{align*}
h'(x)
&>
\frac45-\frac{4}{25x}
-\frac{2(x^2-1)}{x(5x+2)}
-\frac1x\\
&>
\frac25-\frac{29}{25x}>0
\qquad(x\geq\sqrt{15}),
\end{align*}
where the second inequality uses
\[
\frac{2(x^2-1)}{x(5x+2)}<\frac25.
\]
Thus it is enough to check the left endpoint. By the binomial theorem,
\[
\left(1+\frac{2}{5\sqrt{15}}\right)^{14}
>
\frac{293141}{140625}
+\frac{13412}{28125}\sqrt{15}.
\]
The right-hand side is larger than $\sqrt{15}$, since this is equivalent to
\[
\frac{293141}{73565}>\sqrt{15},
\]
and
\[
293141^2-15\cdot73565^2=4754507506>0.
\]
It follows that $h(\sqrt{15})>0$, and hence that
$h(\sqrt n)>0$ for every integer $n\geq15$.
\end{proof}

Lemma~\ref{lem:scalar} is equivalent to
\[
\left(\sqrt n+\frac25\right)^{n-1}>n^{n/2},
\]
and therefore
\begin{equation}\label{eq:failure}
a^{n-1}<n^{-n/2}.
\end{equation}
Combining \eqref{eq:delta}, \eqref{eq:failure}, and the comparison of signed sums in the previous section proves the negative assertion in Theorem~\ref{thm:main} for every $n\geq15$. The positive assertion for $n\leq14$ is the result of~\cite{Pinasco2023}. This completes the proof.

\section{Nonmaximal signed sums in the umbrella family}\label{sec:signed-rescue}

Theorem~\ref{thm:main} concerns the longest signed sum, whereas the question of Pappas and R\'ev\'esz allows an arbitrary choice of signs. The simplex umbrellas make this distinction especially transparent. By symmetry, all rib sign patterns with the same imbalance give the same normalized polarization product, so the discrete problem reduces to one scalar parameter. We first derive the exact formula, then describe the whole natural scale $q$ of order $\sqrt n$, and finally select one particularly simple representative that works in every dimension $n\geq15$.

Throughout this section,
\[
m=n-1,
\qquad
a=\frac1{\sqrt n+2/5},
\]
and the vectors $v_1,\dots,v_m,v_n=e$ are those constructed above. Recall that
\begin{equation}\label{eq:lambda-rescue}
\lambda=\frac{1-ma^2}{m-1}>0,
\qquad
\langle v_i,v_j\rangle=-\lambda
\quad(i\neq j\leq m).
\end{equation}

\subsection{Reduction to the rib imbalance}

Choose signs $\varepsilon_1,\dots,\varepsilon_m\in\{-1,1\}$ on the ribs and keep the positive sign on the handle. Set
\begin{equation}\label{eq:q-general}
q=\sum_{i=1}^m\varepsilon_i,
\end{equation}
so that $q\equiv m\pmod2$, and define
\begin{equation}\label{eq:S-q}
S_q=e+\sum_{i=1}^m\varepsilon_i v_i.
\end{equation}
There are $(m+q)/2$ positive rib signs and $(m-q)/2$ negative ones. The symmetry of the regular simplex implies that every quantity below depends on the sign pattern only through $q$.

The handle factor is
\begin{equation}\label{eq:handle-factor}
\langle S_q,e\rangle=1+aq.
\end{equation}
For a rib $v_j$,
\begin{align*}
\langle S_q,v_j\rangle
&=a+\varepsilon_j-\lambda\sum_{i\neq j}\varepsilon_i\\
&=a-\lambda q+(1+\lambda)\varepsilon_j.
\end{align*}
Introduce
\begin{equation}\label{eq:theta-rescue}
\theta=\frac{a-\lambda q}{1+\lambda}.
\end{equation}
Then
\begin{equation}\label{eq:rib-factors-theta}
\langle S_q,v_j\rangle
=
\begin{cases}
(1+\lambda)(1+\theta),&\varepsilon_j=1,\\[1mm]
-(1+\lambda)(1-\theta),&\varepsilon_j=-1.
\end{cases}
\end{equation}
Finally,
\begin{align}
\|S_q\|^2
&=n+2aq-\lambda(q^2-m)\notag\\
&=n+2aq+\lambda(m-q^2).
\label{eq:norm-Sq}
\end{align}
Consequently,
\begin{equation}\label{eq:P-rescue-new}
P\left(\frac{S_q}{\|S_q\|}\right)
=
\frac{
|1+aq|(1+\lambda)^m
|1+\theta|^{(m+q)/2}
|1-\theta|^{(m-q)/2}
}{
\bigl(n+2aq+\lambda(m-q^2)\bigr)^{n/2}
}.
\end{equation}
It will be convenient to measure the product relative to the polarization threshold by
\begin{equation}\label{eq:F-def}
F_n(q)
=
n^{n/2}
P\left(\frac{S_q}{\|S_q\|}\right).
\end{equation}
Thus $F_n(q)\geq1$ is precisely the desired inequality.

\subsection{The \texorpdfstring{$\sqrt n$}{sqrt(n)} imbalance scale}

The parameter $a$ satisfies $a\sim n^{-1/2}$. Hence an imbalance $q$ of order $\sqrt n$ changes the handle factor $1+aq$ by an amount of order one, while $q/(n-1)\to0$, so the positive and negative rib populations remain asymptotically balanced. This scale therefore retains the symmetry of the ribs while producing a visible tilt in the handle direction.

The next proposition describes the whole scale at once.

\begin{proposition}\label{prop:limiting-profile}
Fix $x\geq0$, and let $(q_n)$ be any sequence of admissible rib imbalances such that
\[
q_n=x\sqrt n+O(1),
\qquad
q_n\equiv n-1\pmod2.
\]
Then
\begin{equation}\label{eq:limiting-profile}
\lim_{n\to\infty}\log F_n(q_n)
=
\log(1+x)-\frac12.
\end{equation}
\end{proposition}

\begin{proof}
Write $t=\sqrt n$. From $a=1/(t+2/5)$ and \eqref{eq:lambda-rescue},
\begin{equation}\label{eq:basic-asymptotics}
a=\frac1t+O(t^{-2}),
\qquad
\lambda=\frac{4}{5t^3}+O(t^{-4}).
\end{equation}
Since $q_n=xt+O(1)$, these estimates give
\begin{equation}\label{eq:theta-asymptotics}
aq_n=x+o(1),
\qquad
m\lambda=o(1),
\qquad
\theta=\frac1t+O(t^{-2}).
\end{equation}
In particular,
\begin{equation}\label{eq:theta-limits}
m\theta^2\longrightarrow1,
\qquad
q_n\theta\longrightarrow x,
\qquad
m\theta^4\longrightarrow0.
\end{equation}
For all sufficiently large $n$ one has $1+aq_n>0$ and $|\theta|<1$, so the absolute values in \eqref{eq:P-rescue-new} may be omitted in the logarithmic calculation below. Moreover, \eqref{eq:norm-Sq} gives
\begin{equation}\label{eq:norm-asymptotic}
\|S_{q_n}\|^2
=n+2x+o(1),
\end{equation}
since $\lambda(m-q_n^2)=o(1)$.

Taking logarithms in \eqref{eq:P-rescue-new} yields
\begin{align}\label{eq:Phi-profile}
\log F_n(q_n)
={}&
\log(1+aq_n)
+m\log(1+\lambda)
+\frac m2\log(1-\theta^2)\notag\\
&+\frac{q_n}{2}\log\frac{1+\theta}{1-\theta}
-\frac n2\log\left(\frac{\|S_{q_n}\|^2}{n}\right).
\end{align}
The five terms have transparent limits. By \eqref{eq:theta-asymptotics},
\[
\log(1+aq_n)\longrightarrow\log(1+x),
\qquad
m\log(1+\lambda)\longrightarrow0.
\]
Using \eqref{eq:theta-limits} and the standard expansions at the origin,
\[
\frac m2\log(1-\theta^2)\longrightarrow-\frac12,
\qquad
\frac{q_n}{2}\log\frac{1+\theta}{1-\theta}
\longrightarrow x.
\]
Finally, \eqref{eq:norm-asymptotic} gives
\[
\frac n2\log\left(\frac{\|S_{q_n}\|^2}{n}\right)
=
\frac n2\log\left(1+\frac{2x+o(1)}n\right)
\longrightarrow x.
\]
The last two contributions cancel, leaving \eqref{eq:limiting-profile}.
\end{proof}

The limiting profile has a sharp transition inside the $\sqrt n$ scale.

\begin{corollary}\label{cor:asymptotic-signed-rescue}
Let $q_n=x\sqrt n+O(1)$ be admissible. If
\[
x>\sqrt e-1,
\]
then $F_n(q_n)>1$ for all sufficiently large $n$. If
\[
0\leq x<\sqrt e-1,
\]
then $F_n(q_n)<1$ for all sufficiently large $n$.
\end{corollary}

\begin{proof}
By Proposition~\ref{prop:limiting-profile}, the sign of $\log F_n(q_n)$ is eventually the sign of
\[
\log(1+x)-\frac12,
\]
which vanishes exactly at $x=\sqrt e-1$.
\end{proof}

Thus the success of a nonmaximal signed sum is not tied to one isolated choice of signs. There is an entire asymptotic regime of almost balanced patterns that works. The critical value $\sqrt e-1$ belongs to the limiting profile; no special meaning should be attached to any convenient constant chosen strictly above it.

\subsection{A uniform choice of the imbalance}

For the all-dimensional statement we choose an even simpler representative. For every $n\geq15$, let
\begin{equation}\label{eq:choice-q-sqrt}
q_n
=
\min\left\{
q\in\mathbb Z:
q\geq\sqrt n,
\quad
q\equiv n-1\pmod2
\right\}.
\end{equation}
Then
\begin{equation}\label{eq:q-sqrt-bounds}
\sqrt n\leq q_n<\sqrt n+2.
\end{equation}
The same rule is used in every dimension. The proof below is divided at $n=25$ only because the estimates become especially simple once $\sqrt n\geq5$; the ten preceding dimensions are then checked directly from the exact formula \eqref{eq:P-rescue-new}.

\begin{proposition}\label{prop:signed-rescue}
For every $n\geq15$, the imbalance $q_n$ defined by \eqref{eq:choice-q-sqrt} satisfies
\[
F_n(q_n)>1.
\]
Equivalently, the corresponding nonmaximal signed sum satisfies
\[
P\left(\frac{S_{q_n}}{\|S_{q_n}\|}\right)>n^{-n/2}.
\]
\end{proposition}

\begin{proof}
Write $q=q_n$ and set
\[
u=a-\lambda(q-1),
\qquad
v=a-\lambda(q+1).
\]
The signed scalar products of $S_q$ with the handle, a positive rib, and
the absolute value of a negative rib are respectively
\[
1+aq,\qquad 1+u,\qquad 1-v.
\]
Moreover, $0\leq v\leq u\leq a<1$.  Indeed, only the first inequality
requires verification.  Put $t=\sqrt n$.  From $q<t+2$ and
\[
\lambda=\frac{20t+29}{(5t+2)^2(t^2-2)}
\]
it is enough to observe that
\[
a-\lambda(t+3)
=
\frac{25t^3-10t^2-139t-107}
{(5t+2)^2(t^2-2)}>0
\qquad(t\geq\sqrt{15}).
\]

For $|y|\leq a$, the power series for the logarithm gives
\[
\log(1+y)\geq y-\frac{y^2}{2(1-a)}.
\]
Apply this inequality to the rib factors and use
$\log z\leq z-1$ for the normalizing factor in
\eqref{eq:P-rescue-new}.  Since there are
$N_\pm=(m\pm q)/2$ ribs of each sign,
\[
N_+u-N_-v=aq+\lambda(m-q^2),
\]
whereas $N_+u^2+N_-v^2\leq ma^2$.  After cancellation of the term
$aq$ we obtain the central estimate
\begin{equation}\label{eq:Phi-short-lower}
\log F_n(q)
\geq
\log(1+aq)
+\frac{\lambda}{2}(m-q^2)
-\frac{ma^2}{2(1-a)}.
\end{equation}

We first use \eqref{eq:Phi-short-lower} for $n\geq25$.  Thus
$t\geq5$ and $t\leq q<t+2$.  The elementary bounds
\[
\log(1+aq)\geq\frac{50}{79},
\qquad
\frac{ma^2}{2(1-a)}<\frac{51}{100},
\qquad
\frac{\lambda(q^2-m)}2<\frac1{10}
\]
follow respectively from
$aq\geq25/27$, the exact formula for $a$, and
$\lambda<t^{-3}$ together with
$q^2-m<4t+5$.  Consequently,
\[
\log F_n(q)
>
\frac{50}{79}-\frac{51}{100}-\frac1{10}
=\frac{181}{7900}>0.
\]

It remains to certify the dimensions $15\leq n\leq24$.  To avoid using
decimal approximations as proof, replace the logarithm in
\eqref{eq:Phi-short-lower} by
\[
\log(1+y)\geq\frac{2y}{2+y}\qquad(y\geq0).
\]
Direct substitution of $a=1/(\sqrt n+2/5)$ and of the admissible values
\[
\begin{array}{c|cccccccccc}
n&15&16&17&18&19&20&21&22&23&24\\
\hline
q_n&4&5&6&5&6&5&6&5&6&5
\end{array}
\]
gives, in each case, the exact algebraic inequality
\[
\frac{2aq_n}{2+aq_n}
+\frac{\lambda}{2}(m-q_n^2)
-\frac{ma^2}{2(1-a)}
>\frac1{10}.
\]
This proves $F_n(q_n)>1$ in the remaining dimensions and completes the
proof.
\end{proof}

For comparison, the longest signed sum corresponds to the extreme imbalance $q=m=n-1$ and fails throughout the umbrella family constructed in Section~3. Proposition~\ref{prop:signed-rescue} shows that the very same configurations contain successful signed sums whose rib signs are nearly balanced, with imbalance only of order $\sqrt n$. Thus the failure of the longest prescription is not a failure of signed sums as such, even within the counterexample family.

\section{Remarks and the signed-sum selection problem}

The vectors of the simplex umbrella span a hyperplane, since
\[
\sum_{i=1}^{n-1}v_i-(n-1)a_n v_n=0.
\]
No linear-independence assumption is involved in Theorem~\ref{thm:main}; the point of the construction is instead that it is produced directly in every dimension and that the longest signed sum is separated from all competitors by a strict gap. Since both this maximality gap and the failure of the polarization bound are strict, sufficiently small generic perturbations yield linearly independent counterexamples with the same properties.

It is important to distinguish the longest-sum prescription from the broader signed-sum selection problem. This distinction already appears in the counterexample of Matolcsi and Mu\~noz: although their longest signed sum fails, they observed that a different, nonmaximal signed sum still satisfies the polarization inequality~\cite[remarks following Theorem~3, item~(1)]{MatolcsiMunoz2006}. The simplex umbrellas exhibit the same phenomenon in every dimension $n\geq15$. The signed-sum selection problem is the following.

\begin{question}[Signed-sum selection problem]\label{question:any-signs}
Given unit vectors $v_1,\dots,v_n\in\mathbb R^n$, do there always exist signs
\[
\varepsilon_1,\dots,\varepsilon_n\in\{-1,1\}
\]
such that
\[
s_\varepsilon
=
\sum_{i=1}^n\varepsilon_i v_i
\neq0
\]
and
\[
\prod_{i=1}^n
\left|
\left\langle
\frac{s_\varepsilon}{\|s_\varepsilon\|},v_i
\right\rangle
\right|
\geq
n^{-n/2}?
\]
\end{question}

For the longest-sum prescription, Theorem~\ref{thm:main} gives a complete answer: it works universally if and only if $n\leq14$. For arbitrary systems, Question~\ref{question:any-signs} has an affirmative answer for $n\leq14$, since the longest signed sum works in that range, and remains open for $n\geq15$. Proposition~\ref{prop:signed-rescue} shows that the counterexamples constructed here do not obstruct Question~\ref{question:any-signs}: in every dimension $n\geq15$, the same simplex umbrellas contain nonmaximal signed sums satisfying the polarization bound.

\section*{Acknowledgments}
The author thanks Daniel Galicer for many stimulating discussions concerning the problem studied in this paper.

\section*{Declaration on the use of generative AI}
During the preparation of this manuscript, the author used OpenAI's ChatGPT as a generative-AI tool for exploratory calculations, symbolic consistency checks, organization of arguments, and drafting and editorial assistance. All AI-assisted material retained in the manuscript was independently reviewed and verified by the author, who assumes full responsibility for the accuracy, originality, and integrity of the work.

\end{document}